\documentclass{amsart}
\usepackage[centertags]{amsmath}
\usepackage{amsfonts}
\usepackage{amssymb}
\usepackage{amsthm}
\usepackage{newlfont}
= 0cm \evensidemargin = 0cm \textheight = 23cm \textwidth = 16cm

\newcommand{\C}{\mathbb C}

\newcommand{\A}{\mathcal A}

\numberwithin{equation}{section}
\newtheorem{ex}{Example}[section]

\newtheorem{cor}{Corollary}[section]
\newtheorem{thm}{Theorem}[section]

\newtheorem{prop}{Proposition}[section]
\theoremstyle{remark}

\def\cit{\hbox{\it I\hskip -5.5pt C\/}}
\title{A formula for the numerical range of elementary operators on $C^*-$algebra}
\author{M. Barraa}

\email{barraa@hotmail.com}

\address{Department de Mathematics, Faculty of Sciences
Semlalia, P.O. Box 2390,  Marrakesh, Morocco.}

\begin{document}

\begin{abstract}
Let $\A$ be a $C^*-$algebra with unit element $1$ and unitary group $U.$  Let $a=(a_1, ..., a_k)$ and $b=(b_1, ..., b_k)$ 
two $k-$tuples of elements in $\A.$ The elementary operator associated to  $a$ and $b$ 
is defined by $ R_{a, b} (x)= \sum_{i=1}^ {k} a_ixb_i.$ In this paper we prove  the following formula for the numerical range of $R_{a,b}$:
$$V(R_{a, b}, B(\A))= [\cup \{ V(\sum_{i=1}^ {k}u^*a_iub_i, \A): u\in U\} ]^-.$$
This formulat solves the problem 4.5 of \cite{Fi}.
\end{abstract}

\maketitle
\section{Introduction}
Let  $\A$ be a complex unital Banach algebra. Let $a=(a_1, ..., a_k)$ and $b=(b_1, ..., b_k)$ 
two $k-$tuples of elements in $\A.$ The elementary operator associated to  $a$ and $b$ 
is defined by:
$$ R_{a, b} (x)= \sum_{i=1}^ {k} a_ixb_i.$$
This is a bounded linear operator on  $\A.$  We refer to \cite{Fi, CM} for good survey of this class of operators.\\
The numerical range of  $a \in \A$ is defined by:
$$ V(a, \A) = \{f(a) : \;\;\; f\in S \},$$
where  $S $ is the set of states in  $\A$ ( $ S  = \{f \in A^*, \;\; \|f\|= 1= f(e) \}$). We refer to 
\cite{BD1, BD2, GR} for the basic facts about numerical ranges.\\
A fondamental example is the $C^*-$algebra $B(H)$ of bounded linear operators on a complex Hilbert 
space $H.$ In this case we write $R_{A,B}$ for the elementary operator defined by 
$$R_{A,B}(X)=\sum_{i=1}^ {k} A_iXB_i.$$
The numerical range of a generalized derivation in $B(H)$ was studied by severals authors, see for instances \cite{Ky, Sha, Mat}. The numerical range of an elementary operators was studied by Seddik \cite{Se1, Se2,Se3}.
In \cite{Fi}, L. Fialkow posed the following problem :\\
\textbf{Problem.} Determine the numerical range and the essential numerical range of the elementary operator $R_{A,B}.$\\
In \cite{B}, we gave a formula for the numerical range of an elementary operator in the case of the operator algebra $B(H).$ 
In this paper we extend this formula to the context of a $C^*-$algebra.

\section{An inclusion in Banach algebra}

An element $u \in \A$ is said to be unitary if $u$ is invertible and $\|u\|=\|u^{-1}\|=1.$ Note that 
$\|ua\|=\|au\|=\|a\|$ for any $a\in A$ and any $u$ unitary in $\A.$
The set of unitary elements is denoted by $U(\A)$ or simply $U.$
\begin{thm}\cite{St}\label{t1}.
 Let $\A$ be a complex Banach algebra with unit. Then for any $a \in \A$ we have
$$ V(a, \A) = \bigcap_{z \in \cit } \{\lambda : \;\;\; |\lambda - z |\leq \| a-z\parallel \}.$$
\end{thm}
Note that $V(u^{-1}au, \A) =V(a, \A)$ for any $a \in \A$ and any unitary $u.$\\
\begin{prop}\label{p1}
Let $\A$ be a complex unital algebra. For $a=(a_1, ..., a_k)$ and $b=(b_1, ..., b_k)$ 
two $k-$tuples of elements in $\A,$ we have the following inclusion:
$$ \cup \{ V(\sum_{i=1}^ {k}u^ {-1}a_iub_i, \A) : u \in U\} \subset V(R_{a, b}, B(\A)).$$
\end{prop}
\begin{proof}
The norm of an elementary operator is defined by 
$$ \| R_{a, b} \| = Sup \{ \| R_{a, b}( x) \| : \;\;\;\| x \| \leq 1 \}.$$
From theorem \ref{t1} we deduce 
$$ V(R_{a,b}, B(\A)) = \bigcap_{z \in \C} \{\lambda : \;\;\; |\lambda - z |\leq \| R_{a,b}-z\| \}.$$
And $$ V(\sum_{i=1}^ {k}u^ {-1}a_iub_i, \A) = \bigcap_{z \in \C} \{\lambda : \;\;\; |\lambda - z |\leq \|\sum_{i=1}^ {k}u^{-1}a_iub_i-z\| \}.$$
which can be written 
$$ V(\sum_{i=1}^ {k}u^ {-1}a_iub_i, \A) = \bigcap_{z \in \C} \{\lambda : \;\;\; |\lambda - z |\leq \|\sum_{i=1}^ {k}a_iub_i-zu\| \}.$$
And hence 
$$ V(\sum_{i=1}^ {k}u^ {-1}a_iub_i, \A) = \bigcap_{z \in \C} \{\lambda : \;\;\; |\lambda - z |\leq \|(R_{a,b}-z)(u)\| \}.$$
But $\|(R_{a, b} -z)(u)\| \leq \|R_{a, b} -z\|,$ thus
$$ \cup \{ V(\sum_{i=1}^ {k}u^ {-1}a_iub_i, \A) : u \in U\} \subset V(R_{a, b}, B(\A)).$$
\end{proof}
If $\A$ is a unital $C^*-$algebra then $u\in \A$ is unitary if and only if $u^*u=uu^*=e.$
\begin{thm} \cite{RD}(Russo-Dy's theorem).\label{t2}
Let $\A$ be a $C^*-$algebra with unit element $1$ and unitary group $U.$ Then the closed unit ball
in $\A$ is the closed convex hull of $U.$
\end{thm}

\begin{cor} \label{c1}
 Let  $a=(a_1, ..., a_k)$ and $b=(b_1, ..., b_k)$ be two k-tuples of elements in  $\A.$ Then
$$ \| R_{a, b} \| = Sup \{ \| R_{a, b}( u) \| : \;\;\; u \in U(\A)\}$$
\end{cor}
\section{Main Result}
The main result of this paper is the following theorem :
\begin{thm}\label{t3}
Let $\A$ be a unital $C^*-$algebra. Let $a=(a_1, ..., a_k)$ and $b=(b_1, ..., b_k)$ 
two $k-$tuples of elements in $\A.$ Then 
$$V(R_{a, b}, B(\A))= [\cup \{ V(\sum_{i=1}^ {k}u^*a_iub_i, \A): u\in U\} ]^-.$$
\end{thm}

\begin{proof}
We need only to show the inclusion "$\subset $". By theorem \ref{c1}
$$ \| R_{a, b} \| = Sup \{ \| \sum_{i=1}^ {k}u^*a_iub_i-zu \| : \;\;\; u \in U(\A)\}.$$
Hence, if $\lambda \in V(R_{a,b}, B(\A)),$ then, for all 
$z\in \C, \;\; \lambda \in \{\mu -z\leq \| R_{a, b}-z \| \}.$\\
Let $\epsilon >0,$ fixed, there exists a unitary $u_{\epsilon} \in \A$ such that 
$$ \| R_{a, b} -z\| \leq \| \sum_{i=1}^ {k}u_{\epsilon}^*a_iu_{\epsilon}b_i-zu \|+\epsilon.$$
Now using theorem \ref{t1}, we get
$$ V(\sum_{i=1}^ {k}u_{\epsilon}^ {-1}a_iu_{\epsilon}b_i, \A) = \bigcap_{z \in \C} \{\lambda : \;\;\; |\lambda - z |\leq \|\sum_{i=1}^ {k}a_iu_{\epsilon}b_i-zu_{\epsilon}\| \}.$$
So, there exists $\mu \in  V(\sum_{i=1}^ {k}u_{\epsilon}^ {-1}a_iu_{\epsilon}b_i, \A),$
such that $  \| \lambda -\mu  \| \leq \epsilon.$
But $\epsilon $ is arbitrary, hence $\lambda \in [\cup_{u \in U(\A)} V(\sum_{i=1}^ {k}u^ {-1}a_iub_i, \A)]^-.$
\end{proof}

\section{Some consequences}
\subsection{Generalized derivation}
We  can get many formulas by putting special elementary operators in our main theorem. 
For example, the generalized derivation defined by $\delta_{a,b}(x)=ax-xb,$ yields the following equality
$$ V( \delta_{a,b}, B(\A))= [\cup\{V(uau^*-b): u \in U\}]^-.$$
In the case of $\A=B(H)$ the algebra of bounded linear operators, it is well known \cite{Sha}, that 
$$ V( \delta_{A,B}, B(\A))=W(A)^--W(B)^-.$$
Which yields 
$$W(A)^--W(B)^- =[\cup\{W(AU-UB):  U \; unitary\; \}]^-.$$
An other case is the multiplication operator $M_{a,b}$ defined by $M_{a,b}(x)=axb.$ In \cite{CM}, Harte
asqued the following question: what is the numerical range of $M_{p,p}$ ? where $p$ is an orthogonal projection $p=p^*=p^2.$ 
From theorem \ref{t2} . we obtain:
$$ V(M_{a,b}, B(\A))= [\cup\{ V(uau^*b, \A) : u \in U(A)]^-.$$
And so $ V(M_{p,p}, B(\A))= [\cup\{ V(upu^*p, \A) : u \in U(A)]^-.$
Hence $M_{p,p}$ is hemitian if and only if $upu^*p$ is hermitian for all $u \in U(\A).$

\subsection{Essential numerical range}
Let $E$ be a complex Banach space and $B(E)$ the Banach algebra of all bounded linear operators acting on $F.$
Denote by $K(E)$ the ideal of all compact operators acting on $E,$ and let $\pi$ be the canonical projection from $B(E)$ 
onto the Calkin algebra $B(E)/K(E).$ Denote further by $ \|.\|_e$ the essential norm $\|T\|_e=inf\{\|T+K\|: K\in K(E)\}.$
The essential numerical range $V_e(T)$ of $T$ is defined by 
$$ V_e(T)= V(\pi(T), B(E)/K(E) , \|.\|_e ).$$
 Let $A=(A_1, ..., A_k)$ and $B=(B_1, ..., B_k)$  two $k-$tuples of elements in $B(E).$
 From Propsition \ref{p1}, we get 
 $$ V_e(\sum_{i=1}^{k}(U^{-1}A_iUB_i)) \subset V(R_{\pi(A), \pi(B)}, B(B(E)/K(E))).$$
 In the case of a Hilbert space $H$ we obtain from theorem \ref{t3}, that 
 $$ V_e(\sum_{i=1}^{k}(U^*A_iUB_i)) = V(R_{\pi(A), \pi(B)}, B(B(H)/K(H))).  \;\;\;(5)$$
 Numerical range of inner derivations $\delta_{\pi(T)}$ were studied by C.K.Fong \cite{Fo}, who proved :
 $$ V(\delta_{\pi(T)})= V_e(T)-V_e(T).$$
 Combining this formula with $(5)$ yields
  $$ V(\delta_{\pi(T)})= V_e(T)-V_e(T)= V_e(U^*TU-T).$$
 

\end{document}